\documentclass[11pt,a4paper]{article}

\usepackage{caption}
\usepackage{amsfonts}
\usepackage{amssymb}
\usepackage{mathrsfs}
\usepackage{amsmath}
\usepackage{amsthm}
\usepackage{enumerate}
\usepackage{graphics}
\usepackage{pict2e}
\usepackage{cite}
\usepackage{subfigure}
\usepackage[subnum]{cases}
\usepackage{multirow}
\usepackage{comment}
\usepackage[dvipsnames]{xcolor}
\usepackage{cleveref}
\usepackage{authblk}
\usepackage{makecell}
\usepackage{float}

\usepackage{graphicx,color}    % for latex
\usepackage{epsfig}
\usepackage{tikz}
\usetikzlibrary{calc,decorations.pathmorphing,decorations.text,fixedpointarithmetic, shadings}
\usepackage{subfigure}
\usepackage{refcount}
\usepackage[hmargin=2.5cm,vmargin=3cm]{geometry}
\usepackage{mathtools, braket}

\usepackage[singlelinecheck=off]{caption}

\usepackage{color}

\makeatletter

\newcommand{\A}{\mathcal{A}}

\renewcommand{\set}[1]{\{#1\}}
\newcommand{\ignore}[1]{}

\newtheorem{theorem}{Theorem}

\newtheorem{corollary}[theorem]{Corollary}

\newtheorem{lemma}[theorem]{Lemma}
\newtheorem{proposition}[theorem]{Proposition}

\makeatletter
\newcommand\figcaption{\def\@captype{figure}\caption}
\newcommand\tabcaption{\def\@captype{table}\caption}
\makeatother

\newtheorem{conjecture}[theorem]{Conjecture}

\newtheorem{remark}[theorem]{Remark}

\title{Fractional balanced chromatic number and \\ arboricity of planar (signed) graphs}

\author[1]{Reza Naserasr}
\author[1]{Lan Anh Pham}
\author[1]{Cyril Pujol}
\author[1,2]{Huan Zhou}

\affil[1]{\small Universit\'e Paris Cit\'e, CNRS, IRIF, F-75013, Paris, France. {Emails: \texttt{\{reza, cpujol, zhou\}@irif.fr}}}
\affil[2]{\small Zhejiang Normal University, Jinhua, China. {Email: \texttt{huanzhou@zjnu.edu.cn}}}

\begin{document}
	
	\setlength\parindent{0pt}
	\maketitle

	\begin{abstract}
		\setlength\parindent{0pt}
		\noindent
		A fractional coloring of a signed graph $(G, \sigma)$ is an assignment of nonnegative weights to the balanced sets (sets which do not induce a negative cycle) such that each vertex has an accumulated weight of at least 1. The minimum total wight among all such colorings is defined to be the fractional balanced chromatic number, denoted by $\chi_{fb}(G, \sigma)$. This value is clearly upper bounded by the fractional arboricity of $G$, denoted $a_f(G)$, where weights are assigned to sets inducing no cycle rather than sets inducing no negative cycle. 
		
		In this work we present an example of a planar signed simple graph of fractional balanced chromatic number larger than 2, thus in particular refuting a conjecture of Bonamy, Kardo\v{s}, Kelly, and Postle suggesting that the fractional arboricity of planar graphs is bounded above by 2. 
		
		By iterating the construction, we show that the supremum of the fractional balanced chromatic number of  planar signed simple graphs is at least as $\frac{83}{41}=2+\frac{1}{41}$. With similar operations, we built a sequence of planar graphs whose limit of fractional arboricity is $a_{f}(G)=2+\frac{2}{25}$.    		
	\end{abstract}

	\section{Introduction}
	
	The four-color theorem has been the driving force behind most developments in graph theory. It is simple, beautiful, and can be stated in numerous ways, each of which offers a different insight and a different direction of study. One of the earliest reformulation, proposed by P.G. Tait in 1884 \cite{Tait84}, led to the development of theory of (nowhere zero) flows and established connection with the Hamiltonicity of cubic bridgeless planar graphs. More precisely, Tait conjectured that every $3$-connected cubic planar graph is Hamiltonian, noting that such cubic graphs are, in particular, $3$-edge-colorable. Prior to the conjecture, he had presented the first reformulation of the four-color conjecture: every cubic bridgeless planar graph is $3$-edge-colorable. 
	
	It took about six decades until W. Tutte \cite{Tutte46} disproved this strong conjecture by building a $3$-connected cubic planar graph which is not Hamiltonian. Tutte's example was composed of three identical pieces, which are now called Tutte's fragment.
	
	The dual of Tutte's fragment is known as Wenger's graph and is used by various authors to disprove various potential extensions or relaxations of the four-color theorem. Wenger, in \cite{Wegner73}, used this gadget to build a planar graph which cannot be partitioned into a bipartite graph and a forest. In this work we similarly use Wenger's gadget to build some examples of planar (signed) graphs whose fractional balanced chromatic number and whose fractional arboricity is strictly larger than 2. To that end we should first introduce these two notions.
	
	\subsection{Vertex arboricity}
	Considering acyclic sets instead of independent sets as color classes we have the notion of \emph{arboricity} of a graph: $a(G)$ is the minimum number of acyclic subsets of $V(G)$ covering all vertices of $G$. One of the key conjectures in this subject is a conjecture of Albertson and Berman \cite{AB79} which claims every planar graph $G$ has an acyclic set of order at least $\frac{|V(G)|}{2}$. 
	
	\emph{Fractional arboricity} of $G$, denoted $a_f(G)$, is the minimum of $\frac{p}{q}$ such that one can assign (at least) $q$ colors to each vertex using only a total of (at most) $p$ colors such that each color class is acyclic. With the size of maximum acyclic set being limited by $\frac{|V(G)|}{a_f(G)}$ and that this is normally a fair enough of an estimate, Bonamy, Kardo\v s, Kelly and Postle conjectured in \cite{BKKP20} that:    
	
	\begin{conjecture}\label{conj:a_f}
		Every planar graph satisfies $a_f(G)\leq 2$. 
	\end{conjecture}
	
	A well known result of Borodin \cite{Borodin76}, proving a conjecture of Gr\"unbaum, is that every planar graph admits a 5-coloring where each pair of colors induces a forest. 
	Given such a 5-coloring $\varphi(G)$, a $\frac{10}{4}$-coloring of $G$ is obtained as follows : the $2$-subsets of $\set{1,2,3,4,5}$ are the colors. A vertex $v$ is assigned the $4$ colors : $\set{\set{\varphi(v),i}\mid i \neq \varphi(v)}$.
	
	\begin{proposition}
		For every planar graph $G$, $a_f(G) \leq 2.5$.
	\end{proposition}

	\subsection{Balanced coloring}
	A different approach for extending four-color theorem, specially in relation to minor theory, comes from signed graphs. This has implicitly been in the literature since at least 1970's where Catlin \cite{Catlin79} proved that odd-$K_4$-free graphs are 3-colorable. In 1980's, Gerard's and Seymour, independently, proposed, as a conjecture, an extension of Catlin's result which would also extend Hadwiger's conjecture claiming that every graph with no odd-$K_t$ minor is $(t-1)$-colorable.
	
	To approach this sort of problems by induction, recently various notions of colorings of signed graphs have been considered. 
	
	We recall that a \emph{signed graph} $(G,\sigma)$ is a	graph $G$ together with the signature function $\sigma$ which assigns to each edge  a sign: positive ($+$) or negative ($-$). A set of vertices is said to be \emph{balanced} if it induces no negative cycle. 
	A reformulation of a conjecture of M\'a\v{c}ajov\'a, Raspaud, \v{S}koveira from 2016, \cite{MRS16}, is to claim that every signed planar (simple) graph can be partitioned into two balanced sets.
	
	This conjecture, was refuted by Kardo\v s and Narboni in \cite{KN21}. They translated the problem to a dual version and used Tutte's fragment to build a counterexample to the dual notion. A direct proof of the claim, using Wenger's graph, is given in \cite{NP22}.

	Given a signed graph $(G, \sigma)$, a \emph{balanced $(p,q)$-coloring}, if exists, is an assignment of $q$ colors to each vertex of $G$ from a platter of $p$ colors such that each colors class induces a balanced set, i.e., it induces no negative cycle. Balanced $(p,1)$-coloring is simply called \emph{balanced $p$-coloring}. In this exact formulation, it was first introduced by T. Zaslavsky \cite{Zaslavsky87}, who has named the corresponding parameter ``\emph{balanced partition number}''. A big part of the theory, and in particular the above mentioned example of Kardo\v s and Narboni in \cite{KN21} are developed under the name ``$0$-free'' coloring. In short, a \emph{$0$-free coloring} of a signed graph $(G, -\sigma)$ using color set $\{\pm1, \pm2, \cdots, \pm p\}$ is the same as balanced $p$-coloring of $(G,\sigma)$.

	The \emph{fractional balanced chromatic number} of a signed graph $(G, \sigma)$, denoted $\chi_{fb}(G, \sigma)$, is $\min {\frac{p}{q}}$ among integers $p$ and $q$ for which $(G,\sigma)$ admits a balanced $(p,q)$-coloring. The value exists if and only $(G, \sigma)$ has no negative loop. This notion is introduced only recently in \cite{JMNNQ24+} and \cite{KNWYZZ25+}.
	
	Assume $\mathcal{S}$ is a signed graph family such that every signed graph in $\mathcal{S}$ has no negative loop. We define the fractional balanced chromatic number of $\mathcal{S}$ as 
	
	$$\chi_{fb}(\mathcal{S})=\sup \{ \chi_{fb}(\hat{G})\mid \hat{G}\in \mathcal{S}  \}.$$
	
	Let $\mathcal{SP}$ be the planar signed simple graph family. As each forest is in particular a balanced set we have $\chi_{fb}(\mathcal{SP})\leq a_f(\mathcal{P}) \leq 2.5$, where $a_f(\mathcal{P})$ is the supremum among all fractional arboricity of planar graphs.
	
	Here using the signed version of Wenger's graph as a key gadget, we build a planar signed (simple) graph for which $\chi_{fb}(G, \sigma)>2$. The underlying graph of this construction, then, is a counterexample to Conjecture~\ref{conj:a_f}. 
	
	We first present a simple proof that our construction has fractional balanced chromatic number bigger than 2. We then calculate the actual fractional balanced chromatic number of our construction to tend to $2+\frac{1}{41}$ and its fractional arboricity to be $2+\frac{2}{31}$.

	\section{Example of planar signed graph satisfying $\chi_{fb}(\hat{G})>2$}
	We need to build a signed planar simple graph which admits no $(2k, k)$-coloring for any integer $k$. We start with the following definition as a property.\newline
	
	{\bf Triangle property.} Given a (planar) signed graph $(G, \sigma)$, a negative triangle of $(G,\sigma)$ is said to have the triangle property if in any balanced $(2k,k)$-coloring, each color appears at least in one of its vertices. A positive triangle is said to have the triangle property if each color appears at most on two of its vertices. 
	
	\begin{lemma}\label{lem:TriangleProperty}
		Given a signed plane graph $(G, \sigma)$, there exists a plane extension $(G', \sigma')$ such that in every balanced $(2k,k)$-coloring of $(G', \sigma')$, each facial triangle of $G$ has the triangle property. 
	\end{lemma} 
	
	\begin{proof}
		If $T$ is a facial negative triangle, then we may add a new vertex $u_T$ to the face joining it to all three vertices and choose a signature such that the induced signed $K_4$ is switching equivalent to $(K_4, -)$. 
		
		As each triple of vertices of $(K_4,-)$ induces a negative triangle, in any $(2k,k)$-coloring of $(K_4,-)$ each color appears on at most on two vertices. But then it follows from a basic counting that each color appears exactly twice. Hence, the colors that appear on $u_T$ appear only once on $T$ and the others appear twice on $T$. 
		
		Thus we may assume the property is validated for any facial negative triangle, and use it to enforce the property on facial positive triangles. To this end, we consider a positive facial triangle and we assume it is switched so that all edges are positive. Then we complete this face as in Figure~\ref{fig:gadget1}, and assume each negative facial triangle admits the triangle property. We now observe that if a color $c$ appears in all three of the $u_i$, then it can appear in none of the $u_j'$ which is not possible.  
	\end{proof}

	\begin{figure}[ht]
		\centering
		\captionsetup{justification=centering}
		\begin{tikzpicture}[scale=.3, rotate=-15]
			
			\definecolor{custom_blue}{rgb}{0.02, 0.46, 1}
			
			% Node styles
			\tikzstyle{none}=[inner sep=0mm]
			\tikzstyle{small black empty}=[fill=white, draw=black, shape=circle, scale=.7pt]
			
			% Edge styles
			\tikzstyle{edge red}=[-, fill=none, draw=red]
			\tikzstyle{blue_edge}=[-,dashed, draw=custom_blue]
			
			\draw[rotate=60] (9,9)  node[small black empty] (x){$u_1$};
			\draw[rotate=180] (9,9)  node[small black empty] (y){$u_2$};
			\draw[rotate=300] (9,9)  node[small black empty] (z){$u_3$};
			\draw[rotate=00] (2,2)  node[small black empty] (u){$u_2'$};
			\draw[rotate=120] (2,2) node[small black empty] (v){$u_3'$};
			\draw[rotate=240] (2,2) node[small black empty] (w){$u_1'$};

			\draw [blue_edge] (x) -- (u);
			\draw [blue_edge] (y) -- (v);
			\draw [blue_edge] (z) -- (w);
			
			\draw [edge red] (v) -- (u);
			\draw [edge red] (u) -- (w);
			\draw [edge red] (w) -- (v);
			
			\draw [blue_edge] (x) -- (y);
			\draw [blue_edge] (y) -- (z);
			\draw [blue_edge] (z) -- (x);
			
			\draw [edge red] (x) -- (v);
			\draw [edge red] (y) -- (w);
			\draw [edge red] (z) -- (u);
			
		\end{tikzpicture}
		\caption{\small Completing the positive triangle}
		\label{fig:gadget1}
		
	\end{figure}
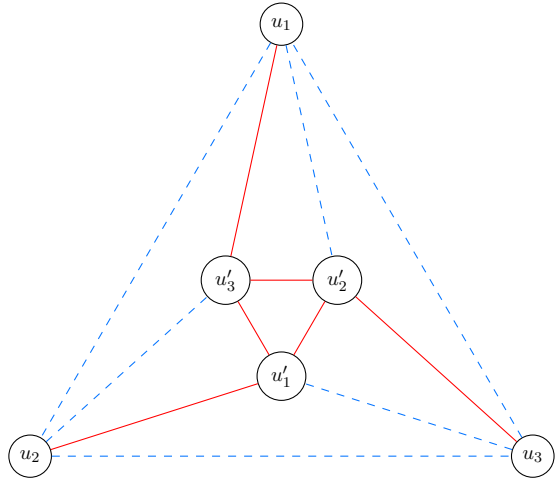

	Assuming every facial triangle must satisfy the triangle property next we build the key gadget.
	
	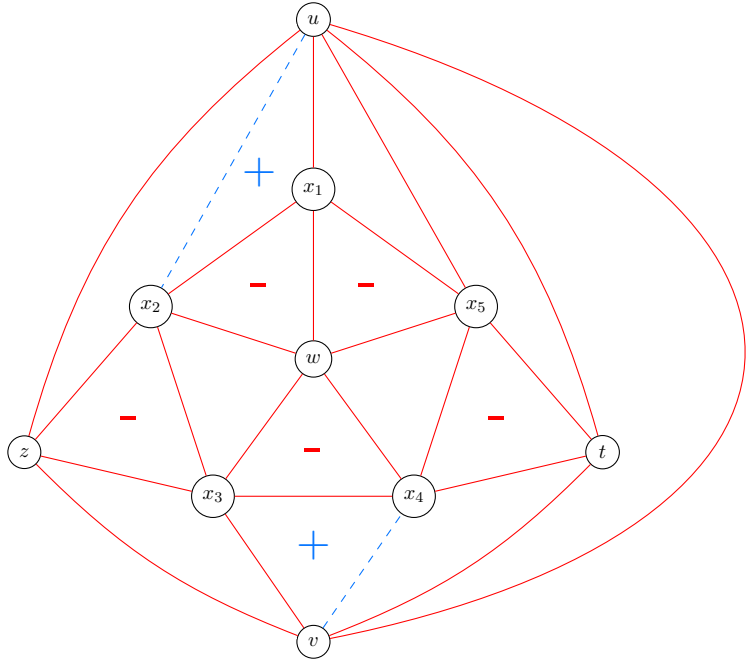
\begin{figure}[ht]
		\centering
		\captionsetup{justification=centering}
		\begin{tikzpicture}[scale=.25, rotate=0]
			
			\definecolor{custom_blue}{rgb}{0.02, 0.46, 1}
			
			% Node styles
			\tikzstyle{none}=[inner sep=0mm]
			\tikzstyle{small black empty}=[fill=white, draw=black, shape=circle, scale=.7pt]
			
			% Edge styles
			\tikzstyle{edge red}=[-, fill=none, draw=red]
			\tikzstyle{blue_edge}=[-,dashed, draw=custom_blue]
			
			\node[small black empty] (0) at (0,0) (w){$w$};
			\draw[rotate=0] (0,9)  node[small black empty] (x1){$x_1$};	
			\draw[rotate=72] (0,9)  node[small black empty] (x2){$x_2$};
			\draw[rotate=144] (0,9)  node[small black empty] (x3){$x_3$};
			\draw[rotate=216] (0,9)  node[small black empty] (x4){$x_4$};
			\draw[rotate=288] (0,9) node[small black empty] (x5){$x_5$};
			\draw[rotate=108] (0,16)  node[small black empty] (z){$z$};
			\draw[rotate=252] (0,16)  node[small black empty] (t){$t$};
			\draw[rotate=0] (0,18)  node[small black empty] (u){$u$};
			\draw[rotate=180] (0,15)  node[small black empty] (v){$v$};

			\node[red, minimum size=8mm, scale=2] at (barycentric cs:z=1,x2=1,x3=1) {-};
			\node[red, minimum size=8mm, scale=2] at (barycentric cs:t=1,x4=1,x5=1) {-};
			\node[red, minimum size=8mm, scale=2] at (barycentric cs:w=1,x1=1,x2=1) {-};
			\node[red, minimum size=8mm, scale=2] at (barycentric cs:w=1,x1=1,x5=1) {-};
			\node[red, minimum size=8mm, scale=2] at (barycentric cs:w=1,x3=1,x4=1) {-};
			
			\node[custom_blue, minimum size=12mm, scale=1.5] at (barycentric cs:u=1,x1=1,x2=1) {+};
			\node[custom_blue, minimum size=12mm, scale=1.5] at (barycentric cs:v=1,x3=1,x4=1) {+};

			\draw [ edge red] (w) -- (x1);
			\draw [ edge red] (w) -- (x2);
			\draw [ edge red] (w) -- (x3);
			\draw [ edge red] (w) -- (x4);
			\draw [ edge red] (w) -- (x5);	
			
			\draw [ edge red] (x5) -- (x1);
			\draw [ edge red] (x1) -- (x2);
			\draw [ edge red] (x2) -- (x3);
			\draw [ edge red] (x3) -- (x4);
			\draw [ edge red] (x4) -- (x5);	
			
			\draw [ edge red] (z) -- (x2);
			\draw [ edge red] (z) -- (x3);
			
			\draw [ edge red] (t) -- (x4);
			\draw [ edge red] (t) -- (x5);
			
			\draw [bend left=12,  edge red] (v) to (z);
			\draw [bend right=12,  edge red] (v) to (t);
			
			\draw [ edge red] (v) -- (x3);
			\draw [blue_edge] (v) -- (x4);

			\draw [edge red ] (u) -- (x1);
			\draw [blue_edge] (u) -- (x2);
			\draw [ edge red] (u) -- (x5);
			\draw [bend left=18,   edge red] (u) to (t);
			\draw [bend right=18,  edge red] (u) to (z);	
			\draw [ edge red] (u) .. controls (30,9) and (30,-9) .. (v);
			\draw [ white] (u) .. controls (-30,9) and (-30,-9) .. (v);
			
		\end{tikzpicture}
		\caption{The signed graph $\widehat{W}$}
		\label{fig:biggadget1}
	\end{figure}
	
	\begin{figure}[ht]
		\centering
		\captionsetup{justification=centering}
		\begin{tikzpicture}[scale=.25, rotate=0]
			
			\definecolor{custom_blue}{rgb}{0.02, 0.46, 1}
			
			% Node styles
			\tikzstyle{none}=[inner sep=0mm]
			\tikzstyle{small black empty}=[fill=white, draw=black, shape=circle, scale=.7pt]
			
			% Edge styles
			\tikzstyle{edge red}=[-, fill=none, draw=red]
			\tikzstyle{blue_edge}=[-,dashed, draw=custom_blue]
			
			\node[small black empty] (0) at (0,0) (w){$w$};
			\draw[rotate=0] (0,9)  node[small black empty] (x1){$x_1$};	
			\draw[rotate=72] (0,9)  node[small black empty] (x2){$x_2$};
			\draw[rotate=144] (0,9)  node[small black empty] (x3){$x_3$};
			\draw[rotate=216] (0,9)  node[small black empty] (x4){$x_4$};
			\draw[rotate=288] (0,9) node[small black empty] (x5){$x_5$};
			\draw[rotate=108] (0,16)  node[small black empty] (z){$z$};
			\draw[rotate=252] (0,16)  node[small black empty] (t){$t$};
			\draw[rotate=0] (0,18)  node[small black empty] (u){$u$};
			\draw[rotate=180] (0,15)  node[small black empty] (v){$v$};
			
			% Add three vertices inside triangle ux1x2
			\node[small black empty,scale=.5pt] at (barycentric cs:u=0.13,x1=0.5,x2=0.43) (a1) {}; 
			\node[small black empty,scale=.5pt] at (barycentric cs:u=0.43,x1=0.2,x2=0.43) (a2) {}; 
			\node[small black empty,scale=.5pt] at (barycentric cs:u=0.43,x1=0.5,x2=0.13) (a3) {}; 
			
			% Add three vertices inside triangle vx3x4
			\node[small black empty,scale=.5pt] at (barycentric cs:v=0.18,x3=0.43,x4=0.43) (b1) {}; 
			\node[small black empty,scale=.5pt] at (barycentric cs:v=0.43,x3=0.18,x4=0.43) (b2) {}; 
			\node[small black empty,scale=.5pt] at (barycentric cs:v=0.43,x3=0.43,x4=0.18) (b3) {}; 
			
			\node[red, minimum size=8mm, scale=2] at (barycentric cs:z=1,x2=1,x3=1) {-};
			\node[red, minimum size=8mm, scale=2] at (barycentric cs:t=1,x4=1,x5=1) {-};
			\node[red, minimum size=8mm, scale=2] at (barycentric cs:w=1,x1=1,x2=1) {-};
			\node[red, minimum size=8mm, scale=2] at (barycentric cs:w=1,x1=1,x5=1) {-};
			\node[red, minimum size=8mm, scale=2] at (barycentric cs:w=1,x3=1,x4=1) {-};
			\node[red, minimum size=8mm, scale=1.5] at (barycentric cs:a1=1,a2=1,a3=1) {-};
			\node[red, minimum size=8mm, scale=1.5] at (barycentric cs:b1=1,b2=1,b3=1) {-};
			
			%\node[custom_blue, minimum size=12mm, scale=1.5] at (barycentric cs:u=1,x1=1,x2=1) {+};
			%\node[custom_blue, minimum size=12mm, scale=1.5] at (barycentric cs:v=1,x3=1,x4=1) {+};

			\draw [ edge red] (w) -- (x1);
			\draw [ edge red] (w) -- (x2);
			\draw [ edge red] (w) -- (x3);
			\draw [ edge red] (w) -- (x4);
			\draw [ edge red] (w) -- (x5);	
			
			\draw [ edge red] (x5) -- (x1);
			\draw [ edge red] (x1) -- (x2);
			\draw [ edge red] (x2) -- (x3);
			\draw [ edge red] (x3) -- (x4);
			\draw [ edge red] (x4) -- (x5);	
			
			\draw [ edge red] (z) -- (x2);
			\draw [ edge red] (z) -- (x3);
			
			\draw [ edge red] (t) -- (x4);
			\draw [ edge red] (t) -- (x5);
			
			\draw [bend left=12,  edge red] (v) to (z);
			\draw [bend right=12,  edge red] (v) to (t);
			
			\draw [ edge red] (v) -- (x3);
			\draw [blue_edge] (v) -- (x4);
			
			\draw[blue_edge] (u)--(a3);
			\draw[edge red] (u)--(a2);
			\draw[edge red] (x1)--(a1);
			\draw[blue_edge] (x1)--(a3);
			\draw[edge red] (x2)--(a1);
			\draw[blue_edge] (x2)--(a2);
			\draw[edge red] (a1)--(a2)--(a3)--(a1);
			
			\draw[blue_edge]    (v)--(b3);
			\draw[edge red]     (v)--(b2);
			\draw[edge red]     (x3)--(b1);
			\draw[blue_edge]    (x3)--(b3);
			\draw[edge red]     (x4)--(b1);
			\draw[blue_edge]    (x4)--(b2);
			\draw[edge red] (b1)--(b2)--(b3)--(b1);
			
			\draw [edge red ] (u) -- (x1);
			\draw [blue_edge] (u) -- (x2);
			\draw [ edge red] (u) -- (x5);
			\draw [bend left=18,   edge red] (u) to (t);
			\draw [bend right=18,  edge red] (u) to (z);	
			\draw [ edge red] (u) .. controls (30,9) and (30,-9) .. (v);
			\draw [ white] (u) .. controls (-30,9) and (-30,-9) .. (v);
			
		\end{tikzpicture}
		\caption{The signed graph $\widehat{W}'$}
		\label{fig:biggadget2}
	\end{figure}
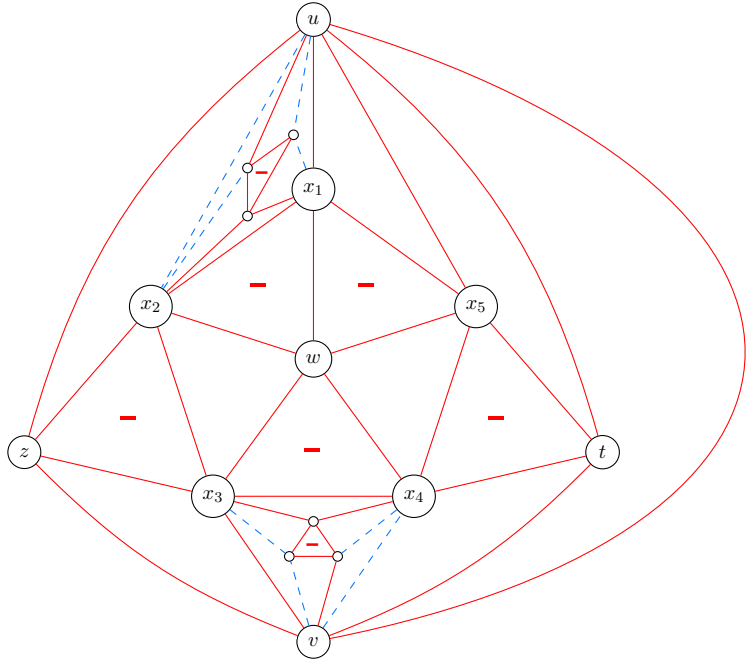

	\begin{lemma}\label{lem:mainGadget}
		In every balanced $(2k,k)$-coloring of the signed graph $\widehat{W}$ of Figure~\ref{fig:biggadget1} satisfying the triangle property for every facial triangle, $u$ and $v$ have no common color. 
	\end{lemma}
	
	\begin{proof}
		Toward a contradiction, assume $\phi$ is such a coloring and that $1\in \phi(u)\cap \phi(v)$. We partition vertices into two parts, part $A$ consisting of vertices that have 1 in their color set, and $\bar{A}$, vertices that do not have the color 1. Observe that $A$ must be a balanced set, and since $uvz$ and $uvt$ are negative triangles, we have $z, t \in \bar{A}$.
		
		We claim that $x_1 \in \bar{A}$. Otherwise, $x_1 \in A$ and first of all by the triangle property on $ux_1x_2$ we have $x_2 \in \bar{A}$ and, secondly,  since $ux_1x_5$ is a negative cycle, we have $x_5 \in \bar{A}$. But then by the triangle property on $zx_2x_3$ and $tx_4x_5$ we have $x_3, x_4 \in A$, however then the triangle $vx_3x_4$ fails the triangle property.
		
		Next we claim that $w\in A$. Otherwise, because of the triangle property on negative triangles $x_1x_2w$ and $x_5x_1w$ we have $x_2, x_5 \in A$. Observe that each of $ux_2x_3v$ and $ux_5x_4v$ induces a negative $4$-cycle, thus we have $x_3, x_4 \in \bar{A}$. But then the triangle $x_3x_4w$ fails the triangle property. 
		
		For the four vertices $x_2,x_3, x_4,x_5$ we then have the following conditions. Of $x_2,x_5$ at least one is in $\bar{A}$ because of the negative 4-cycle $ux_2wx_5$. Of $x_3,x_4$, because of the triangle property on $vx_3x_4$, at least one is in $\bar{A}$. Of $x_2,x_3$ one has to be in $A$ because of the triangle property on $zx_2x_3$ and similarly, because of the triangle $tx_4x_5$ at least one of $x_4, x_5$ should be in $A$. That leaves us with two possibilities: 
		\begin{itemize}
			\item[i.] $x_2, x_4 \in A$ and $x_3, x_5 \in \bar{A}$ 
			\item[ii.] $x_2, x_4 \in \bar{A}$ and $x_3, x_5 \in A$.
		\end{itemize} 
		In case $i.$ $A$ induces negative 5-cycle $ux_2wx_4v$ and in case $ii.$ $A$ induces negative 5-cycle $ux_5wx_3v$. This contradicts the fact that $A$ is a balanced set, thus proving the claim.  		
	\end{proof}
	
	\begin{remark}
		To prove the lemma \ref{lem:mainGadget} we have used the triangle property on the following seven triangles: $ux_1x_2$,$vx_3x_4$, $wx_1x_2$, $wx_1x_5$, $wx_3x_4$, $zx_2x_3$ and $tx_4x_5$, where the first two are positive and five others are negative.
	\end{remark}
	
	\begin{theorem}
		There exists a planar signed simple graph whose fractional balanced chromatic number is strictly larger than $2$. 
	\end{theorem}
	
	\begin{proof}
		Using the mini-gadget of \Cref{fig:gadget1}, we complete the two specified positive triangles of the gadget $\widehat{W}$ of \Cref{fig:biggadget1} to get the gadget $\widehat{W}'$ of \Cref{fig:biggadget2}. For the seven specified negative triangles of $\widehat{W}'$ to satisfy the triangle property, we add a vertex inside each of those faces completing them to have $(K_4, -)$ around each of the added vertices. This results in our final gadget $\widehat{W}''$. To build a final construction $\widehat{G}$, we take a triangle $u_1,u_2,u_3$,   and between each pair of vertices, we put a copy of $\widehat{W}''$ (see \cref{fig:K3+Gadget}, left).
		By the construction $\widehat{G}$ is a simple graph. We have $\chi_{fb}(\widehat{G})>2$ because otherwise it admits a $(2k,k)$-coloring for some $k$ but then, by Lemma~\ref{lem:mainGadget}, we must use disjoint sets of colors for each $u_i$, $u_j$. However, that requires $3k$ colors. 
	\end{proof}
	
	\begin{remark}
		Since we use three copies of the gadget and identify three pairs of vertices, the final graph, which is depicted on the left of Figure~\ref{fig:K3+Gadget}, has $66$ vertices, by identifying the 3 copies of $z$ in the central $6$-face, we get an example on $64$ vertices. 
		Extending the construction on $K_4$, which is depicted in Figure~\ref{fig:K3+Gadget} in two pieces, we will have a graph on $130$ vertices whose fractional balanced chromatic number is calculated in the next section to be $2+\frac{2}{85}$, similarly, by identifying copies of $z$, we may reduce the number of vertices to $127$.    
	\end{remark}

	\begin{figure}[ht]
		\begin{minipage}{.48\textwidth}
			\hspace{-11.cm}
			\resizebox{!}{10cm}{
				\usetikzlibrary{backgrounds}
\usetikzlibrary{arrows}
\usetikzlibrary{shapes,shapes.geometric,shapes.misc}

\pgfdeclarelayer{edgelayer}
\pgfdeclarelayer{nodelayer}
\pgfsetlayers{background,edgelayer,nodelayer,main}

% Node styles
\tikzstyle{none}=[inner sep=0mm]
\tikzstyle{small black empty}=[fill=white, draw=black, shape=circle, scale=.5pt]

% Edge styles
\tikzstyle{edge red}=[-, fill=none, draw=red]
\tikzstyle{blue_edge}=[-, draw={rgb,255: red,6; green,118; blue,255}]

\begin{tikzpicture}[scale = .3]
	\foreach \angle/\index in {1/3, 2/1, 3/2} {
		\begin{scope}[rotate=\angle*120-90,shift={(6.92820,-12)}]
			
			\begin{pgfonlayer}{nodelayer}
				\node [style=small black empty] (0) at (0, 24) {$u_{\angle}$};
				\node [style=small black empty] (1) at (1.25, 15.5) {};
				\node [style=small black empty] (2) at (-3.25, 10.5) {};
				\node [style=small black empty] (3) at (-2.25, 13) {};
				\node [style=small black empty] (4) at (-1.25, 9) {};
				\node [style=small black empty] (5) at (3.75, 9) {};
				\node [style=small black empty] (6) at (4.75, 13) {};
				\node [style=small black empty] (7) at (6.75, 11.25) {};
				\node [style=small black empty] (8) at (0, 0) {$u_{\index}$};
				\node [style=small black empty] (9) at (1.25, 11.5) {};
				\node [style=none] (10) at (-0.5, 17.75) {};
				\node [style=none] (11) at (0, 16.25) {};
				\node [style=none] (12) at (0.75, 18) {};
				\node [style=none] (13) at (-0.75, 18.25) {};
				\node [style=none] (14) at (0, 15.5) {};
				\node [style=none] (15) at (1, 18.75) {};
				\node [style=none] (16) at (0, 17.5) {};
				\node [style=none] (17) at (1, 7) {};
				\node [style=none] (18) at (0.5, 6.25) {};
				\node [style=none] (19) at (1.5, 6.25) {};
				\node [style=none] (20) at (1, 8) {};
				\node [style=none] (21) at (-0.25, 6) {};
				\node [style=none] (22) at (2, 6) {};
				\node [style=none] (23) at (1, 6.5) {};
				\node [style=none] (24) at (0.25, 13) {};
				\node [style=none] (25) at (2.5, 13) {};
				\node [style=none] (26) at (-2.25, 10.75) {};
				\node [style=none] (27) at (1.25, 10) {};
				\node [style=none] (28) at (5.25, 11.25) {};
			\end{pgfonlayer}
			\begin{pgfonlayer}{edgelayer}
				\draw [style=edge red, bend right=15] (0) to (2);
				\draw [style=edge red, bend left=15] (0) to (7);
				\draw [style=edge red, bend left=15] (0) to (1);
				\draw [style=edge red] (0) to (6);
				\draw [style=edge red, bend right=15] (2) to (8);
				\draw [style=edge red] (4) to (8);
				\draw [style=edge red, bend right=15] (8) to (7);
				\draw [style=edge red] (1) to (3);
				\draw [style=edge red] (3) to (4);
				\draw [style=edge red] (4) to (5);
				\draw [style=edge red] (5) to (6);
				\draw [style=edge red] (6) to (1);
				\draw [style=edge red] (1) to (9);
				\draw [style=edge red] (9) to (4);
				\draw [style=edge red] (9) to (3);
				\draw [style=edge red] (3) to (2);
				\draw [style=edge red] (2) to (4);
				\draw [style=edge red] (9) to (5);
				\draw [style=edge red] (5) to (7);
				\draw [style=edge red] (7) to (6);
				\draw [style=edge red] (6) to (9);
				\draw [style={blue_edge}] (5) to (8);
				\draw [style={blue_edge}] (0) to (3);
				\draw [style={blue_edge}] (3) to (10.center);
				\draw [style=edge red] (10.center) to (0);
				\draw [style={blue_edge}] (0) to (12.center);
				\draw [style={blue_edge}] (12.center) to (1);
				\draw [style=edge red] (1) to (11.center);
				\draw [style=edge red] (11.center) to (3);
				\draw [style=edge red] (11.center) to (10.center);
				\draw [style=edge red] (10.center) to (12.center);
				\draw [style=edge red] (12.center) to (11.center);
				\draw [style=edge red] (3) to (14.center);
				\draw [style=edge red] (14.center) to (11.center);
				\draw [style=edge red] (14.center) to (1);
				\draw [style=edge red] (13.center) to (10.center);
				\draw [style={blue_edge}] (13.center) to (3);
				\draw [style=edge red] (13.center) to (0);
				\draw [style={blue_edge}] (0) to (15.center);
				\draw [style=edge red] (15.center) to (12.center);
				\draw [style={blue_edge}] (15.center) to (1);
				\draw [style=edge red] (10.center) to (16.center);
				\draw [style=edge red] (16.center) to (11.center);
				\draw [style=edge red] (16.center) to (12.center);
				\draw [style={blue_edge}] (4) to (17.center);
				\draw [style=edge red] (4) to (18.center);
				\draw [style=edge red] (18.center) to (8);
				\draw [style={blue_edge}] (8) to (19.center);
				\draw [style=edge red] (19.center) to (18.center);
				\draw [style=edge red] (18.center) to (17.center);
				\draw [style=edge red] (17.center) to (19.center);
				\draw [style=edge red] (19.center) to (5);
				\draw [style={blue_edge}] (5) to (17.center);
				\draw [style=edge red] (17.center) to (23.center);
				\draw [style=edge red] (23.center) to (18.center);
				\draw [style=edge red] (23.center) to (19.center);
				\draw [style=edge red] (22.center) to (19.center);
				\draw [style={blue_edge}] (22.center) to (8);
				\draw [style=edge red] (22.center) to (5);
				\draw [style={blue_edge}] (5) to (20.center);
				\draw [style=edge red] (20.center) to (17.center);
				\draw [style={blue_edge}] (20.center) to (4);
				\draw [style=edge red] (4) to (21.center);
				\draw [style=edge red] (21.center) to (18.center);
				\draw [style=edge red] (21.center) to (8);
				\draw [style=edge red] (3) to (26.center);
				\draw [style=edge red] (26.center) to (2);
				\draw [style=edge red] (26.center) to (4);
				\draw [style=edge red] (27.center) to (4);
				\draw [style=edge red] (27.center) to (9);
				\draw [style=edge red] (27.center) to (5);
				\draw [style=edge red] (5) to (28.center);
				\draw [style=edge red] (28.center) to (6);
				\draw [style=edge red] (28.center) to (7);
				\draw [style=edge red] (6) to (25.center);
				\draw [style=edge red] (25.center) to (9);
				\draw [style=edge red] (25.center) to (1);
				\draw [style=edge red] (1) to (24.center);
				\draw [style=edge red] (24.center) to (3);
				\draw [style=edge red] (24.center) to (9);
				\draw [style=edge red, bend right=45, looseness=1.50] (8) to (0);
			\end{pgfonlayer}
	\end{scope}}
\end{tikzpicture}
			}
		\end{minipage}
		\begin{minipage}{.4\textwidth}
			\centering
			\vspace{-2cm}
			\hspace{-2.5cm}
			\resizebox{!}{7cm}{
				\usetikzlibrary{backgrounds}
\usetikzlibrary{arrows}
\usetikzlibrary{shapes,shapes.geometric,shapes.misc}

\pgfdeclarelayer{edgelayer}
\pgfdeclarelayer{nodelayer}
\pgfsetlayers{background,edgelayer,nodelayer,main}

% Node styles
\tikzstyle{none}=[inner sep=0mm]
\tikzstyle{small black empty}=[fill=white, draw=black, shape=circle, scale=.5pt]

% Edge styles
\tikzstyle{edge red}=[-, fill=none, draw=red]
\tikzstyle{blue_edge}=[-, draw={rgb,255: red,6; green,118; blue,255}]

\begin{tikzpicture}[scale = .2]
	\foreach \angle in {1, 2, 3} {
		\begin{scope}[rotate=\angle*120-120]
			\begin{pgfonlayer}{nodelayer}
				\node [style=small black empty] (0) at (0, 25) {$u_{\angle}$};
				\node [style=small black empty] (1) at (0, 15.5) {};
				\node [style=small black empty] (2) at (-5.25, 11) {};
				\node [style=small black empty] (3) at (-3.5, 13) {};
				\node [style=small black empty] (4) at (-2.75, 9) {};
				\node [style=small black empty] (5) at (2.75, 9) {};
				\node [style=small black empty] (6) at (3.5, 13) {};
				\node [style=small black empty] (7) at (5.5, 11) {};
				\node [style=small black empty] (8) at (0, 0) {};
				\node [style=small black empty] (9) at (0, 11.5) {};
				\node [style=none] (10) at (-2, 18) {};
				\node [style=none] (11) at (-1.25, 16.25) {};
				\node [style=none] (12) at (-0.5, 18) {};
				\node [style=none] (13) at (-2.5, 18.25) {};
				\node [style=none] (14) at (-1.25, 15.5) {};
				\node [style=none] (15) at (0, 19) {};
				\node [style=none] (16) at (-1.25, 17.5) {};
				\node [style=none] (17) at (0, 7.25) {};
				\node [style=none] (18) at (-0.75, 6) {};
				\node [style=none] (19) at (0.75, 6) {};
				\node [style=none] (20) at (0, 8.25) {};
				\node [style=none] (21) at (-1.5, 5.75) {};
				\node [style=none] (22) at (1.5, 5.75) {};
				\node [style=none] (23) at (0, 6.5) {};
				\node [style=none] (24) at (-1, 13) {};
				\node [style=none] (25) at (1.25, 13) {};
				\node [style=none] (26) at (-4, 11) {};
				\node [style=none] (27) at (0, 10) {};
				\node [style=none] (28) at (4, 11) {};
			\end{pgfonlayer}
			\begin{pgfonlayer}{edgelayer}
				\draw [style=edge red, bend right=15] (0) to (2);
				\draw [style=edge red, bend left=15] (0) to (7);
				\draw [style=edge red, bend left=15] (0) to (1);
				\draw [style=edge red] (0) to (6);
				\draw [style=edge red, bend right=15] (2) to (8);
				\draw [style=edge red, bend right=15] (4) to (8);
				\draw [style=edge red, bend right=15] (8) to (7);
				\draw [style=edge red] (1) to (3);
				\draw [style=edge red] (3) to (4);
				\draw [style=edge red] (4) to (5);
				\draw [style=edge red] (5) to (6);
				\draw [style=edge red] (6) to (1);
				\draw [style=edge red] (1) to (9);
				\draw [style=edge red] (9) to (4);
				\draw [style=edge red] (9) to (3);
				\draw [style=edge red] (3) to (2);
				\draw [style=edge red] (2) to (4);
				\draw [style=edge red] (9) to (5);
				\draw [style=edge red] (5) to (7);
				\draw [style=edge red] (7) to (6);
				\draw [style=edge red] (6) to (9);
				\draw [style={blue_edge}, bend left=15] (5) to (8);
				\draw [style={blue_edge}, bend right=15] (0) to (3);
				\draw [style={blue_edge}] (3) to (10.center);
				\draw [style=edge red] (10.center) to (0);
				\draw [style={blue_edge}] (0) to (12.center);
				\draw [style={blue_edge}] (12.center) to (1);
				\draw [style=edge red] (1) to (11.center);
				\draw [style=edge red] (11.center) to (3);
				\draw [style=edge red] (11.center) to (10.center);
				\draw [style=edge red] (10.center) to (12.center);
				\draw [style=edge red] (12.center) to (11.center);
				\draw [style=edge red] (3) to (14.center);
				\draw [style=edge red] (14.center) to (11.center);
				\draw [style=edge red] (14.center) to (1);
				\draw [style=edge red] (13.center) to (10.center);
				\draw [style={blue_edge}] (13.center) to (3);
				\draw [style=edge red] (13.center) to (0);
				\draw [style={blue_edge}] (0) to (15.center);
				\draw [style=edge red] (15.center) to (12.center);
				\draw [style={blue_edge}] (15.center) to (1);
				\draw [style=edge red] (10.center) to (16.center);
				\draw [style=edge red] (16.center) to (11.center);
				\draw [style=edge red] (16.center) to (12.center);
				\draw [style={blue_edge}] (4) to (17.center);
				\draw [style=edge red] (4) to (18.center);
				\draw [style=edge red] (18.center) to (8);
				\draw [style={blue_edge}] (8) to (19.center);
				\draw [style=edge red] (19.center) to (18.center);
				\draw [style=edge red] (18.center) to (17.center);
				\draw [style=edge red] (17.center) to (19.center);
				\draw [style=edge red] (19.center) to (5);
				\draw [style={blue_edge}] (5) to (17.center);
				\draw [style=edge red] (17.center) to (23.center);
				\draw [style=edge red] (23.center) to (18.center);
				\draw [style=edge red] (23.center) to (19.center);
				\draw [style=edge red] (22.center) to (19.center);
				\draw [style={blue_edge}] (22.center) to (8);
				\draw [style=edge red] (22.center) to (5);
				\draw [style={blue_edge}] (5) to (20.center);
				\draw [style=edge red] (20.center) to (17.center);
				\draw [style={blue_edge}] (20.center) to (4);
				\draw [style=edge red] (4) to (21.center);
				\draw [style=edge red] (21.center) to (18.center);
				\draw [style=edge red] (21.center) to (8);
				\draw [style=edge red] (3) to (26.center);
				\draw [style=edge red] (26.center) to (2);
				\draw [style=edge red] (26.center) to (4);
				\draw [style=edge red] (27.center) to (4);
				\draw [style=edge red] (27.center) to (9);
				\draw [style=edge red] (27.center) to (5);
				\draw [style=edge red] (5) to (28.center);
				\draw [style=edge red] (28.center) to (6);
				\draw [style=edge red] (28.center) to (7);
				\draw [style=edge red] (6) to (25.center);
				\draw [style=edge red] (25.center) to (9);
				\draw [style=edge red] (25.center) to (1);
				\draw [style=edge red] (1) to (24.center);
				\draw [style=edge red] (24.center) to (3);
				\draw [style=edge red] (24.center) to (9);
				\draw [style=edge red, bend right=45, looseness=1.25] (8) to (0);
			\end{pgfonlayer}
			
	\end{scope}}
\end{tikzpicture}}
		\end{minipage}
		\captionsetup{justification=centering}
		\caption{Signed planar graphs satisfying $\chi_{fb}>2$ with an extension}
		\label{fig:K3+Gadget}
	\end{figure}
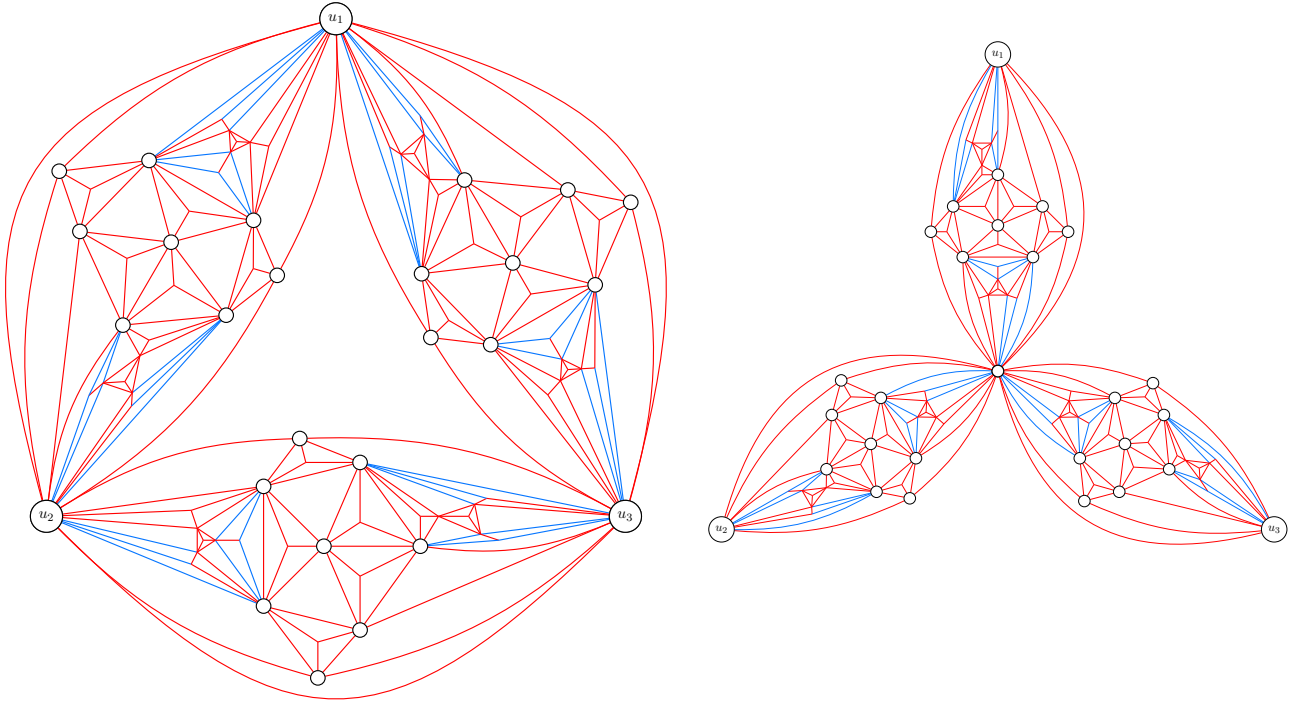

	\section{Improving the lower bound}\label{sec:lowerbound_for_fb}
	
	A better understanding of what make the gadget and the examples work leads to an improved lower bound. To that end we first consider the gadget $\widehat{W}'$ of \Cref{fig:biggadget2}. Observe that there is a sets $\mathcal{T}$ of seven indicated facial negative triangles. The key property of the gadget $\widehat{W}'$ is the following.
	
	\begin{lemma}\label{lem:MissingTriangles}
		Any balanced set of $\widehat{W}'$ containing both $u$ and $v$ misses at least one of the seven triangles of $\mathcal{T}$.    
	\end{lemma} 
	
	\begin{proof}
		Observe that if a balanced set $B$ contains all vertices of a positive triangle, then it cannot contain any of the vertices of the negative triangle of the mini-gadget that might have been built inside it.
		Thus we  will work with $\widehat W$ instead and consider two types of balanced sets containing both $u$ and $v$. The first type is the following two, with the property that each contains one of the two specific positive triangles.
		
		\begin{itemize}
			\item[] $B_1:=\{u,v, x_1, x_2,x_4\}$, contains the positive triangle $ux_1x_2$
			\item[] $B_2:=\{u,v, x_1, x_3,x_4\}$, contains the positive triangle $vx_3x_4$
		\end{itemize}
		
		It can be readily verified that adding any vertex except $x_4$ to $\set{u,v,x_1,x_2}$ induces a negative cycle. Thus $B_1$ is indeed the only maximal balanced set containing $\set{u,v,x_1,x_2}$. Similarly $B_2$ is the only maximal balanced set containing $\set{u,v,x_3,x_4}$.
		
		Next we consider maximal balanced sets distinct from $B_1$ and $B_2$. We show that there are eight of them, labeled $B_3,B_4, \ldots, B_{10}$, each missing one the the specified five negative facial triangles. 
		
		\begin{itemize}
			\item[] $B_3:=\{u,v,w, x_1, x_4\}$, misses the negative triangle $zx_2x_3$
			\item[] $B_4:=\{u,v,w, x_2\}$, misses the negative triangle $tx_4x_5$
			\item[] $B_5:=\{u,v,w, x_3\}$, misses the negative triangle $tx_4x_5$
			\item[] $B_6:=\{u,v,w, x_5\}$, misses the negative triangle  $zx_2x_3$

			\item[] $B_7:=\{u,v, x_1,x_3\}$, misses the negative triangle $tx_4x_5$			
			\item[] $B_8:=\{u,v, x_2, x_5\}$, misses the negative triangle $wx_3x_4$
			\item[] $B_{9}:=\{u,v, x_2,x_4\}$, misses the negative triangle $wx_1x_5$
			\item[] $B_{10}:=\{u,v, x_3, x_5\}$, misses the negative triangle $wx_1x_2$
			
		\end{itemize}
		
		We need to verify that $B_3, \ldots, B_{10}$ are the only maximal balanced sets containing both $u$ and $v$ but not containing the positive triangles $ux_1x_2$ or $vx_3x_4$.  Observe that neither $z$ nor $t$ can be in the set because of the negative triangles $zuv$ and $tuv$. If $w$ is in $B$, then of $x_1, x_2, x_3, x_4$, and $x_5$  no consecutive pair can be in the set. Of the nonadjacent pairs, adding $x_2,x_5$ induces the negative the 4-cycle $ux_2wx_5$, adding $x_2,x_4$ induces the negative 5-cycle $ux_2wx_4v$, adding $x_1,x_3$ induces the negative 5-cycle $ux_1wx_3v$, and adding $x_3,x_5$ induces the negative 5-cycle $ux_5wx_3v$. But adding $x_1,x_4$ induces a positive 5-cycle which is the maximal balanced set $B_3$. And adding each of $x_2$, $x_3$, and $x_5$ results in, respectively, maximal balanced sets  $B_4$, $B_5$, and $B_6$.
		
		Now assume $w\notin B$. From each of the pairs $x_2, x_3$ and $x_4,x_5$ at most one can be in $B$ because of the negative 4-cycles: $ux_2x_3v$ and $ux_5x_4v$. 
		Suppose $x_1\in B$, then $x_5\notin B$ because of the negative triangle $ux_1x_5$. Moreover, since we do not allow $B$ to contain the positive triangle $ux_1x_2$, $x_2\notin B$. Similarly, since we do not allow $B$ to contain the positive triangle $vx_3x_4$, we cannot have both $x_3$ and $x_4$ in $B$. Thus we can either add $x_3$ which results in $B_7$ or we can add $x_4$ which results in $B=\{u,v, x_1,x_4\}$, however being subset of $B_3$, this is not a maximal balanced set of the type we consider. If $x_1$ is not in $B$ either, then we can either add the pair $x_2,x_5$ which results in $B_{8}$, or $x_2,x_4$ which gives us $B_{9}$, or, noting that $vx_3x_4$ is a positive triangle we do not want to have, we may also add $x_3x_5$ to our set and get $B_{10}$.

		All together, we have $10$ (maximal) balanced sets to work with. The first two, $B_1$ and $B_2$, each contain all vertices of one of the two positive triangles we have considered. The balanced sets $B_3,\ldots,B_{10}$ each is missing one of the five specified negative triangles.
	\end{proof}

	This property is already enough to build a planar signed simple graph which does not admits a balanced 2-coloring. Consider a triangle $u_1,u_2,u_3$ and add a copy of $\widehat{W}'$ on each edge identifying that edge with $uv$. If the resulting graph is covered with two balanced sets $B_1$ and $B_2$, then one of the two, say $B_1$, contains an edge, say $u_1u_2$, of the triangle $u_1u_2u_3$. Then in copy $\widehat{W}'_{u_1u_2}$ of $\widehat{W}'$, one of the triangles of $\mathcal{T}$ has no element in $B_1$, but then $B_2$ is not a balanced set. 
	
	To improve on this, we need to strengthen properties of the triangles in $\mathcal{T}$ as follows.
	
	Given integers $p$ and $q$, where $2q \leq p \leq \frac{5q}{2}$, let $m_{p,q}$ be the smallest integer satisfying the following:
	
	\begin{quote}
		There exists a planar signed simple graph $\widehat{T}$ whose outer face is a negative triangle and has the property that in every balanced $(p,q)$-coloring of $\widehat{T}$ at most $m_{p,q}$ colors miss the outer face. 
	\end{quote}
	
	By completing each of the seven triangles in $\mathcal{T}$ of the gadget $\widehat{W}'$ to a suitable copy of $\widehat{T}$ the resulting gadget, denoted  $\widehat{W}^{*}$, has the following property.
	
	\begin{lemma}\label{lem:7Triangles}
		In every balanced $(p,q)$-coloring of $\widehat{W}^*$ we have $|c(u)\cap c(v)|\leq 7m_{p,q}$.
	\end{lemma}
	
	\begin{proof}
		By Lemma~\ref{lem:MissingTriangles} each color class containing both $u$ and $v$, misses at least one of the seven triangles in $\mathcal{T}$. But each of these triangles, being completed to a copy of $\widehat{T}$, can be missed by at most $m_{p,q}$ colors. Thus, the total number of colors classes containing both $u$ and $v$ is bounded above by $7m_{p,q}$. 
	\end{proof}
	
	The main upper bound on the value of $m_{p,q}$ is as follows.
	
	\begin{lemma}\label{lem:BoundingMissingColors}
		Given $p$ and $q$, $2q\leq p \leq \frac{5}{2}q $, we have $m_{p,q}\leq 2p-4q$.
	\end{lemma}
	
	\begin{proof}
		In fact we prove the stronger statement that in a $(p,q)$-coloring of $(K_4,-)$, the total number $l$ of the colors each missing on at least one of the triangles is at most $2p-4q$. Obviously, $m_{p,q}\leq l$. To observe the claim, let $\phi$ be a $(p,q)$-coloring of $(K_4,-)$. Since each of the $p-l$ colors must appear at most twice, we have $4q\leq 2(p-l)+l$, which implies $m_{p,q} \leq l\leq 2p-4q$.
	\end{proof}
	
	Combining this with the following lower bound on $m_{p,q}$, we get an improved lower bound on the fractional balanced chromatic number of planar graphs.
	
	Let $\widehat{G}^{*}$ be the signed graph obtained from $K_3$ by replacing each edge with a copy of $\widehat{W}^{*}$ and consider it with a planar embedding where the outer face is a negative cycle. We have:
	
	\begin{lemma}\label{lem:m<p-3q+21m}
		In every balanced $(p,q)$-coloring of $\widehat{G}^{*}$ at most $p-3q+21m_{p,q}$ colors miss the outer face. 
	\end{lemma}  
	
	\begin{proof}
		Let $y_1$, $y_2$, and $y_3$ be the three vertices of the original triangle. Observe that the number of colors repeating between $y_i$ and $y_j$, $1\leq i<j\leq 3$, is at most $7m_{p,q}$ because of the copy of the gadget $\widehat{W}^{*}$ connecting them. Since no color can appear in all three, we have at least $3q-21m_{p,q}$ distinct colors on the vertices $y_1$, $y_2$, and $y_3$. Thus the number of missing colors is at most $p-3q+21m_{p,q}$.
	\end{proof}
	
	By taking $\widehat{G}^{*}$ of this lemma as an example of $\widehat{T}$ for the definition of $m_{p,q}$ we have the following corollary.
	
	\begin{corollary}\label{cor:m<p-3q+21m}
		For any integer $p$ and $q$ satisfying $2q\leq p \leq \frac{5}{2}q$ we have $m_{p,q}\leq p-3q+21m_{p,q}$.
	\end{corollary}
	
	\begin{theorem}\label{thm:XbfPlanar}
		The fractional balanced chromatic number of signed planar simple graphs is at least $2+\frac{1}{41}$.
	\end{theorem}
	
	\begin{proof}
		By \Cref{lem:BoundingMissingColors} we have $m_{p,q}\leq 2p-4q$ and by \Cref{cor:m<p-3q+21m} we have $m_{p,q}\leq p-3q+21m_{p,q}$. In other words $m_{p,q} \leq \frac{3q-p}{20}$. Combining the inequalities we have $\frac{p}{q}\geq \frac{83}{41}$.
	\end{proof}
	
	We shall note that the lower bound of $\frac{83}{41}$ provided here has the notion of limit and we do not have a concrete example of planar signed simple graph whose fractional balanced chromatic is $2+\frac{1}{41}$. What we have is a sequence of planar singed simple graphs $\widehat{G}_0, \widehat{G}_1, \widehat{G}_2, \ldots$ where $\widehat{G}_0$ is just $(K_4,-)$. To build the rest of the sequence we first define a signed graph $\widehat{U}$ built from $K_4$ by replacing each edge of it with a copy of $\widehat{W}'$. Observe that $\widehat{U}$ has 42 specified triangles each of which is a facial negative triangle. The element  $\widehat{G}_{i}$ of the sequence, then, is built from $\widehat{U}$ by completing each of the 42 specified  triangles with a copy of $\widehat{G}_{i-1}$. 
	
	Intuitively, if $\frac{p}{q}<\frac{83}{41}$, but $p$ and $q$ are relatively large, then the number of colors missing on the outer face of $\widehat G_{i+1}$ among all balanced $(p,q)$-colorings is strictly less than that of $\widehat G_i$. Then invoking \Cref{lem:7Triangles}, we further limit the number of common colors between any two vertices of $\widehat{G}_0$ which is $K_4$. This, in turn, increases the number of required colors.
	Formally, let $\mu_i$ be the maximum number of colors missing on the outer face of $\widehat G_i$, where the maximum is taken among all balanced $(p,q)$-coloring. By \Cref{lem:BoundingMissingColors}, we get: $\mu_0 = 2p-4q$ and  repeating the process of \cref{lem:m<p-3q+21m}, we get: $\mu_{i+1} \leq p-3q+21\mu_i $. Iterating this inequality, we have 
	\begin{align*}
		\mu_i   &\leq \sum_{j=0}^{i-1}21^j (p-3q) + 21^i\mu_0\\
		&\leq \frac{21^i-1}{20}(p-3q) + 21^i\mu_0\\ 	        
		&\leq \frac{21^i-1}{20}(p-3q) + 21^i(2p-4q)\\ 	        
		&\leq 21^i\frac{41p-83q}{20} - \frac{p-3q}{20} .\\ 	        
	\end{align*}
	If $\frac{p}{q} < \frac{83}{41}$, then for $i$ big enough, the right side of the equation becomes negative in which case $\widehat G_i$ does not admit a $(p,q)$-coloring.

	%This insertion will decrease the value of $m_{p,q}$ when $\frac{p}{q}$ is close enough to $2+\frac{1}{41}$, which in turn will enforce the ratio of $\frac{p}{q}$ to be closer to $\frac{83}{41}$. 
	In the above sequence the signed graph $\widehat{G}_{1}$ is the signed graph of \Cref{fig:K3+Gadget}. In the next section we show that $\chi_{fb}(\widehat{G}_{1})=2+\frac{2}{85}$. Furthermore, we show in the next section that each members of this sequence admits an $(83,41)$-coloring, hence this lower bound cannot be improved further by repeated use of the gadget $\widehat{W}''$.
	
	\section{Exact value of $\chi_{fb}(\widehat{G}_1)$ and $\frac{83}{41}$-colorings}
	
	Towards better understanding of the behavior of the sequence $\chi_{fb}(\widehat{G}_i)$ which is used to prove \Cref{thm:XbfPlanar} we provide the exact value for $\chi_{fb}(\widehat{G}_1)$. We then show that the limit of $\chi_{fb}(G_i)$ is $\frac{83}{41}$.

	\begin{theorem}
		For the planar signed simple graph $\widehat{G}_1$ defined in the previous section we have $\chi_{fb}(\widehat{G}_1)=2+\frac{2}{85}$. 
	\end{theorem}
		
	\begin{proof}
		For the lower bound, consider a $(p,q)$-coloring $\phi$ of $\widehat{G}_1$. By \Cref{lem:BoundingMissingColors}, each of the 42 specified triangles of $\widehat{U}$ misses at most $2p-4q$ colors. Applying this on any copy of $\widehat{W}'$, by \Cref{lem:7Triangles} the number of common colors for any two vertices of the main $K_4$ is at most $7(2p-4q)$. On the other hand the number of common colors for at least one pair of vertices of the main $K_4$ is at least $\frac{4q-p}{6}$. Thus we have $7(2p-4q) \geq \frac{4q-p}{6}$, hence $\frac{p}{q}\geq \frac{172}{85}$.
		
		To give a $(172, 85)$-coloring of $\widehat{G}_1$ we first give a $(172, 85)$-coloring of the gadget $\widehat{W}$ of Figure~\ref{fig:biggadget1} such that $|c(u)\cap c(v)|=28$, and that the following triangles satisfy the triangle properties: $ux_1x_2$, $vx_3x_4$ (both positive triangles, having at most 4 colors appearing in all vertices), $wx_1x_2$, $wx_3x_4$, $zx_2x_3$, $tx_4x_5$ (all negative triangles, having at most 4 colors missing). We note that this coloring can be extended to a $(172, 85)$-coloring of $\widehat{W}''$.
		
		This coloring is given in \Cref{tab:color_classes}. In reading this table, when viewed as a $(p,q)$-coloring, the last column shows how many times a balanced set is used as a colors class. Equivalently, this number after being divided by $85$ can be viewed as the weight of the balanced set in a fractional coloring. The first family of balanced sets in the table are the ones containing both $u$ and $v$. Their labeling matches that of \Cref{lem:MissingTriangles} noting that $B_7$ is not used as a color class. The balanced sets $B_{11}, B_{12}, B_{13}$ contain $u$ but not $v$ and $B_{14}, B_{15}, B_{16}$ contain $v$ but not $u$. The remaining ones contain neither. In total we have a list of 23 distinct balanced sets, but together with repetition we have 85. 
		
		\begin{table}[htb]
			\centering
			\begin{tabular}{|c|c|c|c|}
				\hline
				Color classes & Vertices & Description & Repetition  \\
				\hline
				\hline
				$B_1$ & $u,v, x_1, x_2,x_4$ & Contains a positive triangle &$4$\\\Xhline{.5\arrayrulewidth}
				$B_2$ & $u,v, x_1, x_3,x_4$ & Contains a positive triangle &$4$\\ \Xhline{.5\arrayrulewidth}
				$B_3$ & $u,v,w, x_1, x_4$ & 5-cycle &$2$\\ \Xhline{.5\arrayrulewidth}
				$B_4$ & $u,v,w, x_2$ & Forest &$2$\\ \Xhline{.5\arrayrulewidth}
				$B_5$ & $u,v,w, x_3$ & Forest &$2$\\ \Xhline{.5\arrayrulewidth}
				$B_6$ & $u,v,w, x_5$ & Forest &$2$\\ \Xhline{.5\arrayrulewidth}
				%$B_7$ & $u,v, x_1,x_3$ & Forest &$0$\\ \Xhline{.5\arrayrulewidth}
				$B_8$ & $u,v, x_2, x_5$ & Forest &$4$\\ \Xhline{.5\arrayrulewidth}
				$B_{9}$ & $u,v, x_2,x_4$ & Forest &$4$\\ \Xhline{.5\arrayrulewidth}
				$B_{10}$ & $u,v, x_3, x_5$ & Forest &$4$\\ \Xhline{.5\arrayrulewidth}
				\hline
				$B_{11}$ & $u,w, t, x_1,x_3$ & Forest & 49 \\ \Xhline{.5\arrayrulewidth}
				$B_{12}$ & $u,x_2, x_3, x_4,x_5$ & Positive 5-cycle & 4 \\ \Xhline{.5\arrayrulewidth}
				$B_{13}$ & $u, w, z, t, x_2, x_4$ & 6-cycle with a chord & 4 \\ \Xhline{.5\arrayrulewidth}
				$B_{14}$ & $v, z, x_2,x_4, x_5$ & Forest & 49 \\ \Xhline{.5\arrayrulewidth}
				$B_{15}$ & $v,x_1, x_2, x_4,x_5$ & Forest & 4 \\ \Xhline{.5\arrayrulewidth}
				$B_{16}$ & $v, w, z, t, x_2, x_4$ & 6-cycle with a chord & 4 \\ \Xhline{.5\arrayrulewidth}
				$B_{17}$ & $z,t, w, x_1, x_3$ & Forest & 8 \\ \Xhline{.5\arrayrulewidth}
				$B_{18}$ & $z,t, w, x_3, x_5$ & Forest & 6 \\ \Xhline{.5\arrayrulewidth}
				$B_{19}$ & $z,t, x_1, x_3, x_5$ & Forest & 4 \\ \Xhline{.5\arrayrulewidth}
				$B_{20}$ & $z,t, w, x_1, x_4$ & Forest  & 4 \\ \Xhline{.5\arrayrulewidth}
				$B_{21}$ & $z, t, w, x_2, x_5$ & Forest  & 2 \\ \Xhline{.5\arrayrulewidth}
				$B_{22}$ & $z, t, x_1, x_3, x_5$ & Forest & 2 \\ \Xhline{.5\arrayrulewidth}
				$B_{23}$ & $t, x_1, x_2, x_3, x_5$ & Forest  & 2 \\ \Xhline{.5\arrayrulewidth}
				$B_{24}$ & $z, x_1, x_2, x_4, x_5$ & Forest & 2 \\ \Xhline{.5\arrayrulewidth}
			\end{tabular}
			\caption{\centering Description of the coloring of $\widehat{G}_1$}
			\label{tab:color_classes}
		\end{table}
		
		In order to use this coloring of $\widehat{W}''$ for a $(172,85)$-coloring of $\widehat G_1$, we first partition the set $[172]$ of colors to $\{1,2,3,4\}$ and six sets $A_{ij}=A_{ji}$, $1\leq i < j\leq 4$, each of order 28. Then for the copy $W_{ij}$ of $\widehat{W}''$, after a renaming of the above coloring of $\widehat{W}''$, we get a coloring where the 85 colors corresponding to $u_i$ are $\displaystyle \bigcup_{j\neq i} \A_{ij} \cup \{i\}$. % $j\in \{1,2,3,4\}\setminus \{i\}\$.   
	\end{proof}
	
	\vspace{10pt}
	Next we show that the iteration of our gadget, given in the sequence $\widehat{G}_0, \widehat{G}_1, \ldots$ does not pass the threshold of $2+\frac{1}{41}$. In fact we show that any planar signed graph built from the negative triangle by repeated applications of the following two operations admits a balanced $(83,41)$-coloring. Evidently the signed graphs of the sequence $\widehat{G}_i$ are also built in this way. 
	
	\begin{itemize}
		\item[1] Insert a vertex inside a facial negative triangle, connect it to the vertices of the triangle and assign signs in such a way that the resulting signed $K_4$ is switching equivalent to $(K_4,-)$.	
		\item[2] Given an edge $xy$ of a previously built signed graph, add a copy $\widehat{W}'$ of \Cref{fig:biggadget2} such that $xy$ is identified with $uv$.   
	\end{itemize}
	
	\begin{theorem}
	 Every signed graph built from $(K_3,-)$ using operations $[1]$ and $[2]$ admits a balanced $(83,41)$-coloring.
	\end{theorem}
	
	\begin{proof}
		We prove a stronger claim which is easier to carry on by induction. We claim there is a balanced $(83,41)$-coloring of each graph built in this way satisfying extra condition that for each edge $xy$ the number of common colors of $x$ and $y$ is either 13 or 14.  
		
		For the base of induction we consider $\widehat{T}=(K_3, -)$ on vertices $u_1, u_2, u_3$ and, up to a relabeling of the vertices, present three possibilities for a balanced $(83,41)$-coloring of $(K_3, -)$. These colorings are distinguished by the number of common colors, $a_{ij}$, used on vertices $u_i$ and $u_j$.
		
		\begin{itemize}
			\item $(a_{12}, a_{13}, a_{23})=(14,14,14)$
			
			\item$(a_{12}, a_{13}, a_{23})=(14,14,13)$
			
			\item $(a_{12}, a_{13}, a_{23})=(14,13,13)$
		\end{itemize}
	
	Furthermore, we observe that by the extra assumption of $13\leq a_{ij} \leq 14$,  $(a_{12}, a_{13}, a_{23})=(13,13,13)$ is not possible because then the total number of colors would be at least $3\times 41 - (a_{12}+ a_{13} +a_{23})=84$. Thus these three possibilities for coloring $(K_3, -)$ are indeed the only possibilities when the limit of $13$ and $14$ is applied on the number common colors of the end points of an edge.
	
	To prove the inductive claim, first given a coloring $\phi$ of $\widehat{T}$ which is of one of the above three types, we show that there is an extension for operation $[1]$. We call the added vertex $u_4$. Since the four vertices induce a $(K_4,-)$, balanced sets are of size at most 2, so it is enough to present the number of commun colors between pairs of vertices ($a_{ij}$) and the colors appearing on exactly one vertex ($b_i$). 
	
	\begin{itemize}
		\item $(a_{12}, a_{13}, a_{23},a_{14},a_{24},a_{34})=(14,14,14,13,13,13) $, $(b_1,b_2,b_3,b_4) = (0,0,0,2)$
		
		\item $(a_{12}, a_{13}, a_{23},a_{14},a_{24},a_{34})=(14,14,13,13,14,14) $, $(b_1,b_2,b_3,b_4) = (0,0,0,0)$
		
		\item $(a_{12}, a_{13}, a_{23},a_{14},a_{24},a_{34})=(14,14,14,14,14,13) $, $(b_1,b_2,b_3,b_4) = (0,0,2,0)$
	\end{itemize}
	Note that since $u_4$ is only adjacent to $u_1,u_2,u_3$, there is no need to consider the rest of the graph.
	
	For operation [2] we consider signed graph $\widehat{G}$ which is built from $\widehat{T}$ by operations $[1]$ and $[2]$ and let $\widehat{G}'$ be built from $\widehat{G}$ by adding a copy of $\widehat{W}'$ on the edge $xy$ of $\widehat{G}$. By inductive hypothesis, $\widehat{G}$ admits a coloring $\phi$ where end points of each edge, in particular $xy$, have either 14 or 13 colors. Depending on which is the case, we use \Cref{tab:14} or \Cref{tab:13} to extend $\phi$ to a coloring of $\widehat{W}$. We note that in coloring of $\widehat{W}$ given in \Cref{tab:13}, we have $|c(u)\cap c(v)|=13$ but for every other edge the number of common colors on its end points is 14. 
	In coloring of \Cref{tab:13} for every edge the number of common colors on the end points of each edge is 14. 
	Now we just have to extend the coloring of $\widehat{W}$ to a coloring of $\widehat{W}'$. Recall that the difference is just a mini-gadget added in two specified positive triangles (see \Cref{fig:gadget1}).
	
	 Call $u_i$ the vertices of the outside triangle and $u_i'$ the vertices of the inside one, with $u_i$ and $u_i'$ being non-adjacent.
	In both triangles and both tables, the coloring induced on the positive triangle is the same: 
	
	$\set{u_1,u_2,u_3}:2, \set{u_i,u_{i+1}}:12,\set{u_i}:15$ (the indices are computed mod $3$)
	
	Then the coloring can be extended in the following way: 
	
	$\set{u_1,u_2,u_3}:2, \set{u_i,u_{i+1},u_i'}:12,\set{u_i,u_i',u_{i+1}'}:13, \set{u_i,u_{i+1}',u_{i+2}'}:1,\set{u_i,u_{i+2}'}:1$\\
	Note that this extension still verifies the property that every edge is in 14 color classes.
	Thus the condition of induction holds. 
	 		\end{proof}
	
	\begin{table}[H]
    \centering
    \begin{minipage}{0.45\textwidth}
        \centering
\begin{tabular}{|c|c|c|}
        \hline
        Name & Vertices & Repetition  \\
        \hline
        \hline
        $B_1$ & $u,v,x_1,x_2,x_4$ & 1 \\\Xhline{.5\arrayrulewidth}
        $B_2$ & $u,v,x_1,x_3,x_4$ & 2 \\\Xhline{.5\arrayrulewidth}
        $B_3$ & $u,v,x_1,x_3$ & 2 \\\Xhline{.5\arrayrulewidth}
        $B_4$ & $u,v,x_2,x_5$ & 2 \\\Xhline{.5\arrayrulewidth}
        $B_5$ & $u,v,x_2,x_4$ & 2 \\\Xhline{.5\arrayrulewidth}
        $B_6$ & $u,v,x_3,x_5$ & 2 \\\Xhline{.5\arrayrulewidth}
        $B_7$ & $u,v,x_5,w$ & 2 \\\Xhline{.5\arrayrulewidth} \hline
        $B_8$ & $u,x_1,x_3,x_4$ & 5 \\\Xhline{.5\arrayrulewidth}
        $B_9$ & $u,x_1,x_3,x_4,t$ & 1 \\\Xhline{.5\arrayrulewidth}
        $B_{10}$ & $u,x_1,x_4,w,z$ & 1 \\\Xhline{.5\arrayrulewidth}
        $B_{11}$ & $u,x_1,x_4,z,t$ & 2 \\\Xhline{.5\arrayrulewidth}
        $B_{12}$ & $u,x_2,x_3,x_4,x_5$ & 6 \\\Xhline{.5\arrayrulewidth}
        $B_{13}$ & $u,x_2,x_4,x_5$ & 1 \\\Xhline{.5\arrayrulewidth}
        $B_{14}$ & $u,x_2,x_4,w,z,t$ & 2 \\\Xhline{.5\arrayrulewidth}
        $B_{15}$ & $u,x_3,x_5,w$ & 1 \\\Xhline{.5\arrayrulewidth}
        $B_{16}$ & $u,x_3,w,z,t$ & 1 \\\Xhline{.5\arrayrulewidth}
        $B_{17}$ & $u,x_4,w,z,t$ & 8 \\\Xhline{.5\arrayrulewidth}
        $B_{18}$ & $v,x_1,x_2,x_3,x_5,t$ & 7 \\\Xhline{.5\arrayrulewidth}
        $B_{19}$ & $v,x_1,x_2,x_4,x_5$ & 6 \\\Xhline{.5\arrayrulewidth}
        $B_{20}$ & $v,x_1,x_4,w,z,t$ & 1 \\\Xhline{.5\arrayrulewidth}
        $B_{21}$ & $v,x_2,x_4,x_5,z$ & 1 \\\Xhline{.5\arrayrulewidth}
        $B_{22}$ & $v,x_2,x_5,w,z,t$ & 6 \\\Xhline{.5\arrayrulewidth}
        $B_{23}$ & $v,x_2,x_5,w,z$ & 5 \\\Xhline{.5\arrayrulewidth}
        $B_{24}$ & $v,x_2,x_3,x_5$ & 1 \\\Xhline{.5\arrayrulewidth}
        $B_{25}$ & $v,x_4,w,z$ & 1 \\\Xhline{.5\arrayrulewidth}
        $B_{26}$ & $x_2,x_4,w$ & 1 \\\Xhline{.5\arrayrulewidth}
        $B_{27}$ & $x_1,x_3,w,z,t$ & 12 \\\Xhline{.5\arrayrulewidth}
        $B_{28}$ & $x_1,x_3,x_5,z,t$ & 1 \\\Xhline{.5\arrayrulewidth}
    \end{tabular}        \caption{Description of the coloring of $\widehat{W}$ where $u$ and $v$ have at most 13 common colors}
        \label{tab:13}
    \end{minipage}
    \hfill
    \begin{minipage}{0.45\textwidth}
        \centering
        \begin{tabular}{|c|c|c|}
        \hline
        Name & Vertices & Repetition  \\
        \hline
        \hline
        $B_1$ & $u,v,x_1,x_2,x_4$ & 2 \\\Xhline{.5\arrayrulewidth}
        $B_2$ & $u,v,x_1,x_3,x_4$ & 2 \\\Xhline{.5\arrayrulewidth}
        $B_3$ & $u,v,x_1,x_3$ & 2 \\\Xhline{.5\arrayrulewidth}
        $B_{4}$ & $u,v,x_2,x_5$ & 2 \\\Xhline{.5\arrayrulewidth}
        $B_{5}$ & $u,v,x_3,x_5$ & 2 \\\Xhline{.5\arrayrulewidth}
        $B_{6}$ & $u,v,x_2,x_4$ & 2 \\\Xhline{.5\arrayrulewidth}
        $B_{7}$ & $u,v,x_5,w$ & 2 \\\Xhline{.5\arrayrulewidth} \hline
        $B_8$ & $u,x_1,x_3,x_4$ & 4 \\\Xhline{.5\arrayrulewidth}
        $B_9$ & $u,x_1,x_3,x_4,t$ & 1 \\\Xhline{.5\arrayrulewidth}
        $B_{10}$ & $u,x_1,x_4,w,z$ & 1 \\\Xhline{.5\arrayrulewidth}
        $B_{11}$ & $u,x_1,x_4,z,t$ & 2 \\\Xhline{.5\arrayrulewidth}
        $B_{12}$ & $u,x_2,x_3,x_4,x_5$ & 7 \\\Xhline{.5\arrayrulewidth}
        $B_{13}$ & $u,x_2,x_4,w,z,t$ & 1 \\\Xhline{.5\arrayrulewidth}
        $B_{14}$ & $u,x_3,x_5,w$ & 1 \\\Xhline{.5\arrayrulewidth}
        $B_{15}$ & $u,x_3,w,z,t$ & 1 \\\Xhline{.5\arrayrulewidth}
        $B_{16}$ & $u,x_4,w,z,t$ & 9 \\\Xhline{.5\arrayrulewidth}
        $B_{17}$ & $v,x_1,x_2,x_3,x_5,t$ & 7 \\\Xhline{.5\arrayrulewidth}
        $B_{18}$ & $v,x_1,x_2,x_4,x_5$ & 5 \\\Xhline{.5\arrayrulewidth}
        $B_{19}$ & $v,x_1,x_3,x_5,t$ & 1 \\\Xhline{.5\arrayrulewidth}
        $B_{20}$ & $v,x_1,x_4,w,z,t$ & 1 \\\Xhline{.5\arrayrulewidth}
        $B_{21}$ & $v,x_2,x_4,x_5,z$ & 2 \\\Xhline{.5\arrayrulewidth}
        $B_{22}$ & $v,x_2,x_5,w,z,t$ & 5 \\\Xhline{.5\arrayrulewidth}
        $B_{23}$ & $v,x_2,x_5,w,z$ & 6 \\\Xhline{.5\arrayrulewidth}
        $B_{24}$ & $x_1,x_3,x_5,z,t$ & 1 \\\Xhline{.5\arrayrulewidth}
        $B_{25}$ & $x_1,x_3,w,z,t$ & 12 \\\Xhline{.5\arrayrulewidth}
        $B_{26}$ & $x_2,x_4,w$ & 2 \\\Xhline{.5\arrayrulewidth}
    \end{tabular}        \caption{Description of the coloring of $\widehat{W}$ where $u$ and $v$ have at most 14 common colors}
        \label{tab:14}
    \end{minipage}
\end{table}

	\section{Exact value of $a_f$ for the examples}
	
	Let $W_{0}$ be a copy of the (underlying) graph $W$ of Figure~\ref{fig:biggadget1}. Let $W_{1}$ be the (planar) graph built from  $W_{0}$ by replacing the three edges $\set{uz, ux_{1}, ut}$, each with a distinct copy of $W$  where $u$ and $v$ are identified with the endpoints of the edge. 
	
	\begin{theorem}\label{thm-fa}
		We have $a_f(W_{1})=2+\frac{2}{25}$.
	\end{theorem}
	
	\begin{proof}
	 We can verify by basic case analysis that maximum induced forest of $W$ are of order 5. Furthermore, we claim that any forest of $W$ containing both $u$ and $v$ is of order at most 4 (two more vertices). This follows from the list of the maximal balanced sets containing both $u$ and $v$ given in \Cref{sec:lowerbound_for_fb}. Of those ten balanced sets, only three are of size $5$ ($B_{1}, B_{2}, B_{3}$) each of which contain a cycle.
	 
     To prove that $a_f(W_{1})\geq 2+\frac{2}{25}$, let $\phi$ be a $(p,q)$-coloring of $W_{1}$ where every color class is a forest. Let $$a=\max\{|\phi(u)\cap \phi(z)|,|\phi(u)\cap \phi(x_{1})|,|\phi(u)\cap \phi(t)|\}.$$ We first claim $a\leq 5p-10q$.  
     
     Suppose $W_{a}$ is the copy of $W$ corresponding to an edge that gives us $a$. Observe that each color in $\phi(u)\cap \phi(v)$ appears at most on two other vertices of $W_{a}$. Each one in $(\phi(u)\cup \phi(v)) - (\phi(u)\cap \phi(v))$ appears on at most four other vertices. Then the remaining $p-|\phi(u)\cup \phi(v)|$ colors appear on at most $5$ vertices. As each of the $8$ vertices of $V(W_{a})-\set{u,v}$ should receive $q$ colors, we have 
		
		$$ 2a +4(q-a)+4(q-a)+5(p-2q+a)\geq 8q. $$
		
		which implies that $a\leq 5p-10q$. 
		
		To give a lower bound on $a$, we first show that every maximum acyclic set (i.e., of order 5) containing $u$ must include at least two among $z$, $t$, and $x_1$. 
	Assume $A$ is an acyclic set containing $u$ and $z$ but not $t, x_1$ nor $ v$. Then $x_2\notin A$. Thus among four vertices $w, x_3, x_4, x_5$ three should be contained in $A$. As there should be no triangle, either $A$ contains $w, x_3, x_5$ or $x_3, x_4, x_5$, but each, together with $z$ and $u$, induces a $C_5$. The vertex $t$ is symmetric with $z$. Assume $x_1\in A$ and  $z, t, v\notin A$. As $A$ induces no triangle, $x_2, x_5\notin A$. However, the three remaining vertices, $w, x_3, x_4$, induce a triangle. 

		Let $a_{uv}$ be the common colors in $\phi(u)$ and $\phi(v)$ in $W_{0}$ (on which $W_1$ is built). As every acyclic set of order $5$ containing $u$ must include at least two among $z$, $t$, and $x_1$, at most $\frac{3a}{2}$ colors in $\phi(u)-(\phi(u)\cap \phi(v))$ can appear four times in $W_{0}-\{u,v\}$. Applying a more precise calculations on the number of colors on $W_{0}-\set{u,v}$ we conclude: 
	$$2a_{uv}+4\times \frac{3a}{2} +3 \times (q-a_{uv}-\frac{3a}{2})+4(q-a_{uv})+5(p-2q+a_{uv})\geq 8q,$$
	which simplifies to $a \geq \frac{22q-10p}{3}$. Combined with $a\leq 5p-10q$, we get $\frac{p}{q}\geq 2+\frac{2}{25}$. 
	
	To complete the proof, it is enough to give a $(52,25)$-colorings $c$ of $W$ where $$|c(u)\cap c(v)|=|c(u)\cap c(z)|=|c(u)\cap c(x_{1})|=|c(u)\cap c(t)|=10.$$ One such coloring is given in  \cref{tab:forest_coloring}. 		\begin{table}[htb]
                        \centering
                        \begin{tabular}{|c|c|c|}
                                \hline
                                Name & Vertices &  Repetition \\ 
                                \hline
                                \hline                          
                                $F_{0}$ & $u, x_1, x_3, x_4, z$ & 1 \\\Xhline{.25\arrayrulewidth}
                                $F_{1}$ & $u, x_1, x_3, x_4, t$ & 1 \\\Xhline{.25\arrayrulewidth}
                                $F_{2}$ & $u, x_1, x_3, w, t$ & 4 \\\Xhline{.25\arrayrulewidth}
                                $F_{3}$ & $u, x_1, x_4, w, z$ & 4 \\\Xhline{.25\arrayrulewidth}
                                $F_{4}$ & $u, x_4, w, z, t$ & 5 \\\Xhline{.25\arrayrulewidth}
                                $F_{5}$ & $u, v, x_2, x_4$ & 3 \\\Xhline{.25\arrayrulewidth}
                                $F_{6}$ & $u, v, x_2, x_5$ & 4 \\\Xhline{.25\arrayrulewidth}
                                $F_{7}$ & $u, v, x_3, x_5$ & 3 \\\Xhline{.25\arrayrulewidth}
                                $F_{8}$ & $x_1, x_2, x_3, x_4, t$ & 3 \\\Xhline{.25\arrayrulewidth}
                                $F_{9}$ & $x_1, x_2, x_3, x_5, t$ & 2 \\\Xhline{.25\arrayrulewidth}
                                $F_{10}$ & $x_1, x_2, x_3, t, v$ & 2 \\\Xhline{.25\arrayrulewidth}
                                $F_{11}$ & $x_1, x_2, x_4, x_5, v$ & 1 \\\Xhline{.25\arrayrulewidth}
                                $F_{12}$ & $x_1, x_2, x_4, z, v$ & 1 \\\Xhline{.25\arrayrulewidth}
                                $F_{13}$ & $x_1, x_3, x_4, x_5, z$ & 4 \\\Xhline{.25\arrayrulewidth}
                                $F_{14}$ & $x_1, x_3, w, t, v$ & 2 \\\Xhline{.25\arrayrulewidth}
                                $F_{15}$ & $x_2, x_4, x_5, z, v$ & 2 \\\Xhline{.25\arrayrulewidth}
                                $F_{16}$ & $x_2, x_5, w, z, v$ & 4 \\\Xhline{.25\arrayrulewidth}
                                $F_{17}$ & $x_2, x_5, w, t, v$ & 2 \\\Xhline{.25\arrayrulewidth}
                                $F_{18}$ & $x_2, w, z, t, v$ & 1 \\\Xhline{.25\arrayrulewidth}
                                $F_{19}$ & $x_3, x_5, w, z, t$ & 3 \\\Xhline{.25\arrayrulewidth}
                                \hline
                        \end{tabular}
                        \caption{ \centering Description of a coloring of $W$}
                        \label{tab:forest_coloring}
                \end{table}
	\end{proof}
	
	\begin{remark}{\em
	Unlike the case of balanced coloring, here we give a concrete example of a planar graph satisfying $a_f(W_1)=2+\frac{2}{25}$. This the best bound that can be built with our methods using the graph $W$. For each of the value $\gamma \in \{7, 8, 9, 0\}$, one can find a $(52,25)$-coloring of $W$ where each color class is a forest, the number of common colors on $u$ and $v$ is the value $\gamma$, and the number od common colors on two vertices of any other edge belongs to the set $\{7, 8, 9, 0\}$. This allows us to extend the coloring when more copies of $W$ are added to our graph. \\
	It should also be noted that if we consider the graph built from $K_4$ by replacing each edge with a copy of $W$, then the resulting graph has fractional arboricity $2+\frac{2}{31}$.}
	\end{remark}

	\section*{Acknowledgment} This work has received support under the program ``Investissement d'Avenir" launched by the French Government and implemented by ANR, with the reference ``ANR‐18‐IdEx‐0001" as part of its program ``Emergence".
	
	\bibliographystyle{plain}
	\bibliography{references}
	
\end{document}